\documentclass[11pt]{amsart}

\usepackage{mathtools} \mathtoolsset{showonlyrefs,showmanualtags}
\usepackage[
    colorlinks=true,       
    linkcolor=magenta,          
    citecolor=magenta,        
    filecolor=magenta,      
    urlcolor=magenta,           
    ]{hyperref}

\usepackage{stmaryrd,mathrsfs}
\usepackage{xcolor}
\numberwithin{equation}{section}
\allowdisplaybreaks[4]
\usepackage{amsmath,amsthm,amscd,amssymb,amsfonts, amsbsy}
\usepackage{latexsym}
\usepackage{exscale}
\usepackage{mathrsfs,amssymb}
\usepackage{enumerate}

\newtheorem{theorem}[equation]{Theorem}

\newtheorem{cor}[equation]{Corollary}
\newtheorem{lemma}[equation]{Lemma}

\newtheorem{Remark}[equation]{Remark}
\newtheorem{definition}[equation]{Definition}
\newtheorem{con}[equation]{Conjecture}

\numberwithin{equation}{section}

\def\supp{{\rm supp}}

\let \e=\varepsilon

\def\0{{\rm \bf{0}}}

\def\q{{\vec{q}}}

\begin{document}

\title{Weighted mixed weak-type inequalities for multilinear operators}

\author{Kangwei Li}
\address[K. Li]{BCAM, Basque Center for Applied Mathematics, Mazarredo, 14. 48009 Bilbao Basque Country, Spain}
\email{kli@bcamath.org}

\author{Sheldy J. Ombrosi}
\address[S. Ombrosi]{Departamento de Matem\'atica\\
Universidad Nacional del Sur\\
Bah\'ia Blanca, 8000, Argentina}\email{sombrosi@uns.edu.ar}

\author{M. Bel\'en Picardi}
\address[B. Picardi]{Departamento de Matem\'atica\\
Universidad Nacional del Sur\\
Bah\'ia Blanca, 8000, Argentina}\email{belen.picardi@uns.edu.ar}

\begin{abstract}
In this paper we present a theorem that generalizes Sawyer's classic result about mixed weighted inequalities to the multilinear context. Let $\vec{w}=(w_1,...,w_m)$ and $\nu = w_1^\frac{1}{m}...w_m^\frac{1}{m}$, the main result of the paper sentences that under different conditions on the weights we can obtain 
$$\Bigg\| \frac{T(\vec f\,)(x)}{v}\Bigg\|_{L^{\frac{1}{m}, \infty}(\nu v^\frac{1}{m})} \leq C \ \prod_{i=1}^m{\|f_i\|_{L^1(w_i)}},
$$
where $T$ is a multilinear Calder\'on-Zygmund operator. To obtain this result we first prove it for the $m$-fold product of the Hardy-Littlewood maximal operator $M$,  and also for $\mathcal{M}(\vec{f})(x)$: the multi(sub)linear maximal function introduced in \cite{LOPTT}. 

As an application we also prove a vector-valued extension to the mixed weighted weak-type inequalities of multilinear Calder\'on-Zygmund operators.

\end{abstract}

\keywords{mixed weighted inequalities, multilinear operators.}

\thanks{K.L. is supported by Juan de la Cierva-Formaci\'on 2015 FJCI-2015-24547, the Basque Government through the BERC 2014-2017 program and
 Spanish Ministry of Economy and Competitiveness MINECO: BCAM Severo Ochoa excellence accreditation SEV-2013-0323. S. O. and B. P. are supported by CONICET PIP 11220130100329CO, Argentina.}

\maketitle

\section{Introduction}
In 1985, E. Sawyer \cite{S} proved the following mixed weak-type inequality:

\begin{theorem}\label{sawyer}
If $u,v \in A_1$, then there is a constant $C$ such that for all $t>0$,
$$uv\Big(\Big\{ x \in \mathbb{R}  :\; \frac{M(fv)(x)}{v(x)} \; > t \Big\}\Big)
 \; \leq \; \frac{C}{t} \int_{\mathbb{R}} |f(x)|u(x)v(x)\,dx.$$
\end{theorem}

In the same work E. Sawyer conjectured that the previous theorem is valid if the maximal operator is replaced by the Hilbert transform. In 2005, D. Cruz-Uribe, J. M. Martell and C. P\'erez \cite{CMP} extended this result to $\mathbb{R}^n$. Furthermore, they proved it for Calder\'on-Zygmund operators, solving this Sawyer's conjecture.

\begin{theorem}[\cite{CMP}]\label{cmpteo}
If $u,v \in A_1$, or $u\in A_{1}$ and $uv\in A_{\infty}$, then there is a constant $C$ such that for all $t>0$,
$$uv\Bigg(\Big\{ x \in \mathbb{R}^n  :\; \frac{|T(fv)(x)|}{v(x)} \; > t \Big\}\Bigg)
 \; \leq \; \frac{C}{t} \int_{\mathbb{R}^n} |f(x)|u(x)v(x)\,dx,$$ where $T$ is Hardy-Littlewood maximal function or any Calder\'on-Zygmund operator.
\end{theorem}

Quantitative estimates of these mixed weighted results can be found in \cite{OPR}. Moreover, it was conjectured in \cite{CMP}  that the conclusion of the previous theorem still holds if a weaker and more general hypothesis is satisfied. That is, if we have the following conditions on the weights, $u\in A_{1}$ and $v\in A_{\infty}$. Recently, in \cite{LOP} the first two authors and C. P\'erez solved  such conjecture. Namely, the following theorem was proved, which constitutes the more difficult case of this class of mixed weighted inequalities.

\begin{theorem}[\cite{LOP}]\label{LOPinf}
Let $T$ be a Calder\'on-Zygmund operator or the Hardy-Littlewood maximal operator and let $u\in A_1$ and $v\in A_{\infty}$. Then there is a finite constant $C$ depending on the $A_1$ constant of $u$ and the $A_{\infty}$ constant of $v$ such that
\begin{equation} \label{a1ainf}
\ \Big\| \frac{T(fv)}{v}\Big\|_{L^{1, \infty}(uv)} \leq C \|f\|_{L^1(uv)}. 
\end{equation}
\end{theorem}

\bigskip

On the other hand, the study of multilinear Calder\'on-Zygmund theory started in the seventies with the works of R. Coifman and Y. Meyer (\cite{CoMe1} and \cite{CoMe2}). However, a systematic treatment of this topic appears later with works of L. Grafakos and R. Torres \cite{GraTor1, GraTor2}. We recall the definition of a multilinear Calder\'on-Zygmund operator:
let $T:S({\mathbb R}^n)\times\dots\times S({\mathbb R}^n)\to
S'({\mathbb R}^n)$ be a multilinear operator initially defined on the $m$-fold
product of Schwartz spaces and taking values into the space of
tempered distributions;
we say that $T$ is an $m$-linear
Calder\'on-Zygmund operator if, for some $1\le q_1, \dots, q_m<\infty$ and $\frac{1}{m} \leq p < \infty$ satisfying $\frac{1}{p}=\frac{1}{q_1}+\dots+\frac{1}{q_m}$, it
extends to a bounded multilinear operator from
$L^{q_1}\times\dots\times L^{q_m}$ to $L^p$, and if
there exists a function
$K$ defined
off the diagonal $x=y_1=\dots=y_m$ in $({\mathbb
R}^n)^{m+1}$ satisfying the appropriate decay and smoothness conditions (see Page 5 in \cite{LOPTT} ) and such that
$$T(f_1,\dots,f_m)(x) =\int_{\mathbb R^n} \cdots \int_{\mathbb R^n}
K(x,y_1,\dots,y_m) \prod_{i=1}^m f_i(y_i)\, dy_1\cdots dy_m $$
for all $ x \notin \cap_{i=1}^{m} \supp f_{i}$.
Related to weighted estimates for these operators, the first result was obtained in \cite{GraTor2} (see also \cite{PerTor}) where the authors proved that, if $1 < q_1, \dots, q_m < \infty$ and $w$ is a weight in the Muckenhoupt $A_{q_0}$ class for $q_0=\min\{q_1, \dots, q_m\}$, an $m$-linear Calder\'on-Zygmund operator $T$ maps $L^{q_1}(w)\times\dots\times L^{q_m}(w)$ into $L^p(w)$.
 In \cite{LOPTT} Lerner et. al, developed the appropriate class of multiple weights for $m$-linear Calder\'on-Zygmund operators. Now, we recall some of those results in \cite{LOPTT} that will be useful for us along this paper.  Let  $1\le q_1, \dots , q_{m}<\infty$ and $\frac{1}{m} \leq p < \infty$ be such that $\frac{1}{p}=\frac{1}{q_1}+\dots+\frac{1}{q_m}$.
We say that $\vec w=(w_1,\dots,w_m)$ satisfies the \emph{multilinear $A_{\q}$ condition} if
\begin{equation}\label{multiap}
\sup_{Q}\Big(\frac{1}{|Q|}\int_Q\nu_{\vec
w}\Big)^{1/p}\prod_{i=1}^m\Big(\frac{1}{|Q|}\int_Q
w_i^{1-q'_i}\Big)^{1/q'_i}<\infty
\end{equation}
where the supremum is taken over all cubes $Q$
(when $q_i=1$, $\Big(\frac{1}{|Q|}\int_Q w_i^{1-q'_i}\Big)^{1/q'_i}$
is understood as $\displaystyle(\inf_Q w_i)^{-1}$). Now, if $\vec w$ satisfies the $A_{\q}$ condition and $1<q_1,\dots,q_m<\infty$, then an $m$-linear Calder\'on-Zygmund operator $T$ maps $L^{q_1}(w_1)\times \dots \times L^{q_m}(w_m) $ into $ L^{p}(\nu_{\vec w})$. If at least one $q_i=1$, then $T$ maps $L^{q_1}(w_1)\times \dots \times L^{q_m}(w_m) $ into $ L^{p,\infty}(\nu_{\vec w})$. It is shown that $\prod_{i=1}^m A_{q_i} \subseteq A_{\q}$ and that this inclusion is strict. Moreover, if $T$ is the $m$-linear Riesz transform, it was proved in \cite{LOPTT} that $A_{\q}$ is a necessary condition for such weighted estimate of $T$.

One of the key points in \cite{LOPTT} was the introduction of the multi(sub)linear maximal function $\mathcal{M}$ defined by $$\mathcal M(\vec f\,)(x)=\sup_{Q\ni x
}\prod_{i=1}^m\frac{1}{|Q|}\int_Q|f_i(y_i)|dy_i,$$ where $\vec{f}=(f_1,...,f_m)$ and the supremum is taken over all cubes $Q$ containing $x$.

This maximal operator is smaller than the product $\prod_{i=1}^m Mf_i$, which was the auxiliary operator used previously to estimate multilinear singular integral operators.

The aim of this paper is to obtain mixed weighted estimates that generalize Theorem \ref{LOPinf} to the multilinear context. We will investigate both $\prod_{i=1}^m Mf_i$ and $\mathcal{M}(\vec f)$ under different assumptions. Then by an extrapolation theorem  we can  also  prove mixed weighted inequalities for multilinear Calder\'on-Zygmund operators.        

The first result of this paper is the following:

\begin{theorem}\label{main}
Let $w_1,...,w_m \in A_1$ and $v\in A_{\infty}$. Denote $\nu = w_1^\frac{1}{m}...w_m^\frac{1}{m}$. Then,
$$\Bigg\| \frac{\prod_{i=1}^m Mf_i}{v}\Bigg\|_{L^{\frac{1}{m}, \infty}(\nu v^\frac{1}{m})} \leq C \prod_{i=1}^m{\|f_i\|_{L^{1}(w_i)}}.$$
\end{theorem}

The particular case in which the weight $v=1$ in the theorem above was proved in \cite{LOPTT} (See  Theorem 3.12 there). Adding a non-constant function $v$ in the distribution function makes the proof more complicated. However, benefits from Theorem \ref{LOPinf} and the ideas in \cite{LOPTT}  allow us to obtain the result.  It is obvious that the conclusion in Theorem \ref{main} also holds for the maximal operator $\mathcal M$.

We will see below that as a consequence of Theorem \ref{main}, we can obtain the same result for multilinear Calder\'on-Zygmund operators. However, we know that $\prod_{i=1}^m Mf_i$ is too big to estimate multilinear Calder\'on-Zygmund operators. In fact, it was proved in \cite{LOPTT}  that the condition  $w_1,...,w_m \in A_1$ is stronger than $\vec{w}=(w_1,...,w_m) \in A_{\vec{1}}$. And since the last condition characterizes the weak type of a multilinear Calder\'on-Zygmund operator $T$ from 
$L^{1}(w_{1})\times \dots \times L^{1}(w_{m})$ into $L^{1/m,\infty}(\nu)$, it is natural to ask if it is possible to relax the hypothesis  $w_1,...,w_m \in A_1$  in Theorem \ref{main} if we put $T$ or $\mathcal M $ instead of $\prod_{i=1}^m Mf_i$.      
The next theorem gives a partially positive answer.  

\begin{theorem}\label{Max}

Let $\vec{w}=(w_1,...,w_m) \in A_{\vec{1}}$, $\nu = w_1^\frac{1}{m}...w_m^\frac{1}{m}$ and $v$ be a weight satisfying $\nu v^{\frac 1m}\in A_\infty$.  Then there is a constant $C$ such that
\begin{equation}\label{se111}
\Bigg\| \frac{\mathcal M(\vec f\,)(x)}{v}\Bigg\|_{L^{\frac{1}{m}, \infty}(\nu v^\frac{1}{m})} \leq C \ \prod_{i=1}^m{\|f_i\|_{L^1(w_i)}}.
\end{equation}
\end{theorem}

So \eqref{se111} holds for either $\vec{w}=(w_1,...,w_m) \in A_{\vec{1}}$ and $\nu v^{\frac 1m}\in A_\infty$ or (as a consequence of Theorem \ref{main}) if  the weights $w_i \in A_1$ for $i=1,...,m$ and $v \in A_\infty$. These conditions are independent. However, we believe that there is a unified condition that contains both such that \eqref{se111} holds. That is:

\begin{con}\label{con1}
Let $\vec{w}=(w_1,...,w_m) \in A_{\vec{1}}$, $v^{1/m}\in A_{\infty}$ and $\nu = w_1^\frac{1}{m}...w_m^\frac{1}{m}$.  Then there is a constant $C$ such that
$$\Bigg\| \frac{\mathcal M(\vec f\,)(x)}{v}\Bigg\|_{L^{\frac{1}{m}, \infty}(\nu v^\frac{1}{m})} \leq C \ \prod_{i=1}^m{\|f_i\|_{L^1(w_i)}}$$
\end{con}


\begin{Remark}
In general, under the hypothesis of Theorem \ref{Max} the estimate \eqref{se111} does not hold if $\mathcal M(\vec f)$  is replaced by $\prod_{i=1}^m Mf_i$ even in the case that $v(x)=1$.  This fact was proved in Remark 7.5 in \cite{LOPTT}.   

\end{Remark}

The following (extrapolation) theorem allows us to reduce the problem of multilinear Calder\'on-Zygmund operators to the multilinear maximal function, exactly as in the linear case.  Actually, the theorem below was essentially obtained in \cite{OmPe}, which is a combination of  Theorem 1.5  and some observations in Section 2.2 there. However,  for the sake of completeness we will give a complete proof in Appendix A .

\begin{theorem}[\cite{OmPe}]\label{extension}
Let $\vec{w}=(w_1,...,w_m) \in A_{\vec{1}}$, $v^{1/m}\in A_{\infty}$ and $\nu = w_1^\frac{1}{m}...w_m^\frac{1}{m}$.  Then
$$\Bigg\| \frac{T(\vec{f})}{v}\Bigg\|_{L^{\frac{1}{m}, \infty}(\nu v^\frac{1}{m})} \leq C \ \Bigg\| \frac{\mathcal{M}(\vec{f})}{v}\Bigg\|_{L^{\frac{1}{m}, \infty}(\nu v^\frac{1}{m})}$$
where $C$ is a constant and $T$ is a multilinear Calder\'on-Zygmund operator.
\end{theorem}

\bigskip

Now, as a consequence of Theorems \ref{main}, \ref{Max} and \ref{extension} we obtain the main result of this paper:  

\begin{theorem}\label{muczo}
Let $T$ be a multilinear Calder\'on-Zygmund operator, $\vec{w}=(w_1,...,w_m)$ and $\nu = w_1^\frac{1}{m}...w_m^\frac{1}{m}$. Suppose that  $\vec{w} \in A_{\vec{1}}$ and $\nu v^{\frac 1m}\in A_\infty$ or  $w_1,...,w_m \in A_1$ and $v \in A_\infty$.  Then there is a constant $C$ such that
\begin{equation}\label{se}
\Bigg\| \frac{T(\vec f\,)(x)}{v}\Bigg\|_{L^{\frac{1}{m}, \infty}(\nu v^\frac{1}{m})} \leq C \ \prod_{i=1}^m{\|f_i\|_{L^1(w_i)}}
\end{equation}
\end{theorem}

Recall the definition of $RH_{\infty}$:

\begin{definition}\label{RHinf}
We denote by $RH_{\infty}$ the class of weights $w$ such that for all cube $Q$, there exists a constant $C$, which is independent of $Q$, such that $$ \mathop{\textup{ess}\, \sup}_{x\in Q}  \ w(x) \leq \frac{C}{|Q|} \ \int_Q {w(x) dx}.$$
\end{definition}

Since if $u \in A_1$ and $v\in RH_{\infty}$, then $uv^{\frac 1m}\in A_{\infty}$ (see Lemma \ref{pesos} below), we have a direct corollary of Theorem \ref{muczo}.

\begin{cor}
Let $\vec{w}=(w_1,...,w_m)\in A_{\vec{1}}$  and let $v\in RH_{\infty}$. Then there is a constant $C$ such that
$$\Bigg\| \frac{T(\vec f\,)(x)}{v}\Bigg\|_{L^{\frac{1}{m}, \infty}(\nu v^\frac{1}{m})} \leq C \ \prod_{i=1}^m{\|f_i\|_{L^1(w_i)}}$$
\end{cor}

\bigskip

The article is organized as follows. In Section 2 we prove Theorem \ref{main}. The proof of Theorem \ref{Max} is presented in Section 3. In Section 4 as an application of Theorem \ref{muczo}  we obtain a vector-valued extension of the mixed weighted inequalities obtained for multilinear Calder\'on-Zygmund operators.  The last Section is the Appendix A where we give a complete proof of Theorem \ref{extension}.

\section{Proof of Theorem \ref{main}}

First, we need the following Lemma.

\begin{lemma}\label{pesos}
\
\begin{enumerate}[(a)]
\item \label{a} $w\in A_{\infty}$ if and only if $w=w_1w_2$, where $w_1\in A_1$ and $w_2 \in RH_{\infty}$.
\item \label{b} If $w\in A_1$, then $w^{-1} \in RH_{\infty}$.
\item \label{c} If $u,v \in RH_{\infty}$, then $uv\in RH_{\infty}$.
\item \label{Ainf}  If $w\in A_{\infty}$ and $u\in RH_{\infty}$, then $wu\in A_{\infty}$.
\item \label{d} If $w \in RH_{\infty}$, then $w^{s}\in RH_{\infty}$ for any $s>0$.
\end{enumerate}
\end{lemma}

All these properties of $A_p$ classes of Muckenhoupt are well known. The first three can be found in \cite{CUN} or \cite{GCRF} for instance. However,  as far as we know, \eqref{Ainf} and \eqref{d} are not written specifically in any place, so in the following paragraph we present a simple argument for them.

\
\

Proof of (\ref{Ainf}): 
Since $w\in A_{\infty}$, by (\ref{a}), $w=w_1w_2$, where $w_1\in A_1$ and $w_2\in RH_{\infty}$.
By (\ref{c}), $w_2u\in RH_{\infty}$. Then,
$wu=(w_1w_2)u=w_1(w_2u)$. Now, by (\ref{a}), $wu\in A_{\infty}$.

Proof of (\ref{d}): If $s\ge 1$, this is just by H\"older's inequality, so we only need to consider the case $s<1$. Since $w\in RH_\infty\subset A_\infty$, then 
\[
\frac 1{|Q|}\int_Q w \le C_{s,w} \Big(\frac 1{|Q|}\int_Q w^s\Big)^{\frac 1s}.
\]
By definition, our claim follows immediately. 

\noindent \textbf {Proof of Theorem \ref{main} :}
The main idea of this proof is to reduce the problem to the linear case and then apply Theorem \ref{LOPinf}. We define $$E=\{ x: v(x)< \prod_{i=1}^m{Mf_i}(x) \leq 2v(x) \}$$
Let $\tilde{v_i}=\prod_{j=1, j\neq i}^m{(Mf_j)}^{-1}$ and let $v_i=v\tilde{v_i}$.
Observe that $v\in A_{\infty}$ and $\tilde{v_i}\in RH_{\infty}$. By Lemma \ref{pesos} (\ref{Ainf}), $v_i\in A_{\infty}$.
In order to prove the theorem it is enough to show that
$$v^{\frac{1}{m} }\nu (E)\leq C \prod_{i=1}^{m}\|f_i\|_{L^1(w_i)}.$$
By H\"{o}lder's inequality and Theorem \ref{LOPinf}, we have 
\begin{align*}
v^\frac{1}{m} \nu (E) & \leq {\int_E{\big(\prod_{i=1}^{m} Mf_i \ w_i\big)^{\frac{1}{m}}}}\leq \prod_{i=1}^{m}{\Big( \int_E{Mf_i \ w_i} \Big) ^{\frac{1}{m}}}\leq 2 \ \prod_{i=1}^{m}\Big( \int_E {v_i \ w_i}\Big)^{\frac{1}{m}}  \\
&\leq 2 \ \prod_{i=1}^{m}\Big( \int_{\{ x : Mf_i >v_i\}} {v_i \ w_i}\Big)^{\frac{1}{m}}\leq C \ \prod_{i=1}^{m}{\|f_i\|_{L^1(w_i)}^{\frac 1m}},
\end{align*}
where in the last inequality we have used Theorem \ref{LOPinf} since $w_i\in A_1$ and $v_i\in A_{\infty}$.

\section{Proof of Theorem \ref{Max}}
We follow the strategy of \cite{LOP}. So we only need to consider the dyadic multilinear maximal functions. First, recall that if $\vec{w}=(w_1,...,w_m) \in A_{\vec{1}}$, then  $\nu = w_1^\frac{1}{m}...w_m^\frac{1}{m}\in A_1$ (see Theorem 3.6  in \cite{LOPTT}). On the other hand, since  $\nu \in A_1$ it is not difficult to check that the hypothesis $\nu v^{\frac 1m}\in A_\infty$ implies that $v^{\frac 1m}\in A_\infty$.
 
We shall prove 
\[
\nu v^{\frac 1m}\Big(\Big\{x: 1<\frac{\mathcal M_d(f_1,\cdots,f_m)(x)}{v(x)}\le 2\Big\}\Big)\le C \Big(\prod_{i=1}^m\int_{\mathbb R^n} |f_i|w_i\Big)^{\frac 1m}
\] Without loss of generality, we can assume $f_i\ge 0$, $i=1,\cdots,m$.
Let 
\[
E_k:=\Big\{x: 1<\frac{\mathcal M_d(f_1,\cdots,f_m)(x)}{v(x)}\le 2, \,a^{mk}<v(x)\le a^{m(k+1)}\Big\},
\] where $a>2^{n}$. Again, define 
\[
\Omega_k=\{\mathcal M_d(f_1,\cdots,f_m)> a^{mk}\}
\]
and let $\{I_j^k\}_j$ be the collection of maximal dyadic cubes in $\Omega_k$. Then by maximality, $a^{mk}<\prod_{i=1}^m \langle f_i\rangle_{I_j^k}\le 2^{mn}a^{mk}$. Splitting the collection $\{I_j^k\}_j$ to 
\[
\mathcal Q_{l,k}=\{I_j^k: a^{k+l}\le\langle v^{\frac 1m}\rangle_{I_j^k}< a^{k+l+1}\}, \quad l\in \mathbb Z.
\]
Then we have
\begin{align*}
\sum_{k\in \mathbb Z}\nu v^{\frac 1m}(E_k)&=\sum_{k\in \mathbb Z}\nu v^{\frac 1m}(E_k\cap\Omega_k)= \sum_{k\in \mathbb Z}\sum_{j}\nu v^{\frac 1m}(E_k\cap I_{j}^k) 
\\
&\le \sum_{k\in \mathbb Z}\sum_{l\ge 0}\sum_{I_j^k\in \mathcal Q_{l,k}} a^{k+1}\nu(E_k\cap I_j^k)
\\
&= \sum_{k\in \mathbb Z}\sum_{l\ge 0}\sum_{I_j^k\in \Gamma_{l,k}} a^{k+1}\nu(E_k\cap I_j^k),
\end{align*}
where
\[
\Gamma_{l,k}=\{I_{j}^k\in \mathcal Q_{l,k}: |I_j^k\cap \{x: a^k<v^{\frac 1m}\le a^{k+1}\}|>0\}.
\]
 From now on, we shall deal with the case $l=-1$ and $l \ge 0$ separately. By monotone convergence theorem, it suffices to give a uniform estimate for
\[
\sum_{k\ge N} \sum_{l\ge 0}\sum_{I_j^k\in \Gamma_{l,k}} a^{k+1}\nu(E_k\cap I_j^k),
\]
where $N<0$. 
We have the following two lemmas. The proofs are essentially given in \cite{LOP}.
\begin{lemma}
$\Gamma=\cup_{ l\in \mathbb Z}\cup_{k\ge N}\Gamma_{l,k}$ is sparse.
\end{lemma}

\begin{lemma}
For $l\ge 0$ and $I_j^k\in \Gamma_{l,k}$, there exist constants $c_1$ and $c_2$ depending on $\nu,v$ such that
\[
  \nu(E_k\cap I_j^k)\le c_1  e^{- c_2 l} \nu(I_j^k).
\]
\end{lemma}
We also have the following lemma.
\begin{lemma}\label{rh}
If $w_1 w_2 \in A_\infty$, then for any cube $Q$, we have 
\[
\langle w_1 w_2 \rangle_Q \le C([w_1 w_2]_{A_\infty}) \langle w_1 \rangle_Q \langle w_2 \rangle_{Q}.
\]
\end{lemma}
\begin{proof}
let $E_1=\{x\in Q: w_1(x)>  4\langle w_1\rangle_{Q}\}$ and
$E_2=\{x\in Q: w_2(x)> 4\langle w_2\rangle_{Q}\}$. Then by Chebyshev, it is easy to see that $E:=Q\setminus (E_1\cup E_2)$ satisfies $|E|\ge \frac 12|Q|$.
Since $w_1w_2\in A_\infty$, we have
\begin{align*}
w_1w_2(Q)&\le c([w_1 w_2]_{A_\infty}) w_1w_2(E)\le 16 c([w_1 w_2]_{A_\infty}) \langle w_1\rangle_{Q} \langle w_2\rangle_{Q} |E| \\
&\le 16 c([w_1 w_2]_{A_\infty}) \langle w_1\rangle_{Q} \langle w_2\rangle_{Q} |Q|.
\end{align*}
\end{proof}
With this lemma, we can also obtain the exponential decay for $l <0$. 
\begin{lemma}
For $l< 0$ and $I_j^k\in \Gamma_{l,k}$, there exists a constant $c_1$ depending on $\nu,v$ such that
\[
  \nu(E_k\cap I_j^k)\le c_1  a^{l} \nu(I_j^k).
\]
\end{lemma}
\begin{proof}
By Lemma \ref{rh}, we have 
\[
\nu v^{\frac 1m}(I_j^k)\le   C_{\nu,v} \langle \nu\rangle_{I_j^k} \langle v^{\frac 1m}\rangle_{I_j^k} |I_j^k|\le C_{\nu,v} a^{k+l} \nu(I_j^k).
\]
On the other hand, 
\[
\nu v^{\frac 1m}(I_j^k) \ge a^k \nu(E_k \cap I_j^k).
\]
Therefore
\[
\nu(E_k\cap I_j^k)\le C_{\nu,v} a^{l} \nu(I_j^k).
\] 
\end{proof}
Now fix $l$, form the principal cubes for $\cup_{k\ge N}\Gamma_{l,k}$: let $\mathcal P_0^l$ be the maximal cubes in $\cup_{k\ge N}\Gamma_{l,k}$, then for $m\ge 0$, if $I_s^t\in \mathcal P_m^l$, we say $I_j^k\in \mathcal P_{m+1}^l$ if $I_j^k$ is maximal (in the sense of inclusion) in $ \mathcal D(I_s^t)$ such that
\[
\langle \nu \rangle_{I_{j}^k}> 2 \langle \nu \rangle_{I_s^t}
\]
Denote $\mathcal P^l=\cup_{m\ge 0}\mathcal P_m^l$ and $\pi(Q)$ is the minimal principal cube which contains $Q$. We have
\begin{align*}
\sum_k\sum_{l\in \mathbb Z}\sum_{I_j^k\in \Gamma_{l,k}}a^{k+1}\nu(E_k\cap I_j^k)
& \le \sum_{l\in \mathbb Z}c_1  e^{- c_2 |l|} a^{1-l} \sum_k\sum_{I_j^k\in \Gamma_{l,k}} \langle v^{\frac 1m}\rangle_{I_j^k} \nu(I_j^k)\\
&\le\sum_{l\in \mathbb Z}2c_1  e^{- c_2 |l|}a^{1-l} \sum_{I_s^t\in \mathcal P^l} \langle \nu\rangle_{I_s^t} \sum_{k,j:\pi(I_j^k)=I_s^t} v^{\frac 1m}(I_j^k)\\
&{\lesssim_n} \sum_{l\in \mathbb Z}c_1  e^{- c_2 |l|} a^{-l} [v^{\frac 1m}]_{A_\infty}\sum_{I_s^t\in \mathcal P^l} \langle \nu\rangle_{I_s^t} v^{\frac 1m}(I_s^t)\\
&= \sum_{l\in \mathbb Z}c_1  e^{- c_2 |l|} a^{-l} [v^{\frac 1m}]_{A_\infty}\int_{\mathbb R^n} v^{\frac 1m}\sum_{I_s^t\in \mathcal P^l} \langle \nu\rangle_{I_s^t}\chi^{}_{I_s^t}\\
&\lesssim \sum_{l\in \mathbb Z}c_1  e^{- c_2 |l|} a^{-l} [v^{\frac 1m}]_{A_\infty}[\nu]_{A_1}\sum_{Q\in \mathcal P^l_*} \nu v^{\frac 1m}(Q),
\end{align*}
where in the last step we have used the stopping criteria, i.e., 
\[
\sum_{I_s^t\in \mathcal P^l} \langle \nu\rangle_{I_s^t}\chi^{}_{I_s^t}\le 2[\nu]_{A_1} \nu(x), \,\, a.e. \,\,x\in \mathop\cup_{s,t: I_s^t\in \mathcal P^l} I_s^t
\] and  $\mathcal P^l_*$ is the collection of maximal cubes (in the sense of inclusion) in $\mathcal P^l$. For fixed $Q\in \mathcal P^l_*$, by Lemma \ref{rh}
\begin{align*}
\nu v^{\frac 1m}(Q)&\le   C_{\nu,v} \langle \nu\rangle_{Q} \langle v^{\frac 1m}\rangle_{Q} |Q| \lesssim_n a^{l}C_{\nu,v}  u(Q)\prod_{i=1}^m \langle f_i\rangle_{Q}^{\frac 1m}\\
&\le a^{l}C_{\nu,v}[\vec w]_{A_{\vec 1}}  \prod_{i=1}^m (\int_Q f_i w_i)^{\frac 1m}.
\end{align*}
By the disjointness of $Q$ and H\"older's inequality, we obtain 
\[
\sum_{Q\in \mathcal P^l_*} \nu v^{\frac 1m}(Q)\lesssim a^{l}C_{\nu,v}[\vec w]_{A_{\vec 1}} \prod_{i=1}^m \|f_i\|_{L^1(w_i)}^{\frac 1m},
\]
and we conclude the proof.

\section{A vector-valued extension of Theorem \ref{muczo}}

Recently in \cite{CMO} D. Carando, M. Mazzitelli and the second author obtained a generalization of the Marcinkiewicz-Zygmund inequalities to the context of multilinear operators. In the particular case of paraproducts, Marcinkiewicz-Zygmund inequalities were obtained by C. Benea and C. Muscalu in \cite{BM} and \cite{BM2}.  The results in \cite{CMO} extend the previous ones in \cite{GM} and \cite{BPV}. 

The following theorem is one of the results in \cite{CMO}.

\begin{theorem}[\cite{CMO}]\label{CMO}
Let $0<p, q_1, \dots, q_m < r<2$ or $r=2$ and $0<p, q_1, \dots, q_m < \infty$ and, for each $1 \leq i \leq m$, consider $\{f^i_{k_i}\}_{k_i} \subset L^{q_i}(\mu_i)$.  And Let $S$ be a multilinear operator such that  
 $S\colon L^{q_1}(\mu_1) \times \cdots \times L^{q_m}(\mu_m) \to L^{p,\infty}(\nu)$, then, there exists a constant $C>0$ such that 
 
 \begin{equation}\label{weak MZ multilineal p,q>0}
\left\| \left( \sum_{k_1, \dots, k_m} |S(f^1_{k_1}, \dots, f^m_{k_m})|^r \right)^{\frac{1}{r}}  \right\|_{L^{p, \infty}(\nu)} \leq C  \|S\|_{weak} \prod_{i=1}^m \left\| \left( \sum_{k_i} |f^i_{k_i}|^{r} \right)^{\frac{1}{r}} \right\|_{L^{q_i}(\mu_i)}.
\end{equation}

\end{theorem}

As a consequence of this theorem and Theorem \ref{main}  we obtain the following mixed weighted vector valued inequality for a multilinear Calder\'on-Zygmund operator $T$.

\begin{cor}\label{ext vec} 
Let $S(\vec{f})=\frac{T(\vec{f})}{v}$, where $T$ is a Calder\'on-Zygmund operator. Let $\vec{w}=(w_1,...,w_m) \in A_{\vec{1}}$ and $v\in RH_{\infty}$, or $w_1,...,w_m\in A_1$ and $v\in A_{\infty}$. Let $\nu = w_1^\frac{1}{m}...w_m^\frac{1}{m}$ and let $1<r\leq 2$.
For each $1 \leq i \leq m$, consider $\{f^i_{k_i}\}_{k_i} \subset L^{1}(w_i)$. Then, there exists a constant $C>0$ such that 

\begin{equation}\label{ext vect}
\left\| \left( \sum_{k_1, \dots, k_m} |S(f^1_{k_1}, \dots, f^m_{k_m})|^r \right)^{\frac{1}{r}}  \right\|_{L^{\frac{1}{m}, \infty}(\nu v^{\frac{1}{m}})} \leq C \prod_{i=1}^m \left\| \left( \sum_{k_i} |f^i_{k_i}|^{r} \right)^{\frac{1}{r}} \right\|_{L^{1}(w_i)}.
\end{equation}
\end{cor}

Observe that $S$ satisfies $S\colon L^{1}(w_1) \times \cdots \times L^{1}(w_m) \to L^{\frac{1}{m},\infty}(\nu v^{\frac{1}{m}})$, so we are under the hypothesis of Theorem \ref{CMO}.

\bigskip

\section{Appendix A. Proof of Theorem \ref{extension}}

First, as we mentioned before,  that if $\vec{w}=(w_1,...,w_m) \in A_{\vec{1}}$ then $\nu = w_1^\frac{1}{m}...w_m^\frac{1}{m}\in A_1$. 
We will follow the ideas of  \cite[Theorem 1.5]{OmPe}, and these ideas are based in previous one in \cite{CMP}. Define the operator $S$ by

$$Sf(x)=\frac{M(f\nu )(x)}{\nu (x)}$$
when $\nu (x) \neq 0$ and $Sf(x)=0$ when $\nu (x)=0$. (Since $\nu \in A_1$, $\nu > 0$ a.e.).

Since $\nu \in A_1$, $S$ is bounded on $L^{\infty }(\nu v^{\frac{1}{m}})$ with constant $C=[\nu ]_{A_1}$, that is, $$\| Sf\| _{L^{\infty }(\nu v^{\frac{1}{m}})} \leq [\nu ]_{A_1} \| f\| _{L^{\infty }(\nu v^{\frac{1}{m}})}.$$ We will now show that $S$ is bounded on $L^{p_0}(\nu v^{\frac{1}{m}})$ for some $1<p_0<\infty $. Observe that

$$\int_{\mathbb{R}_n} Sf(x)^{p_0} \ \nu (x) \ v^{\frac{1}{m}}(x) \ dx = \int_{\mathbb{R}_n} M(f\nu )(x)^{p_0} \ \nu (x) ^{1-p_0} \ v^{\frac{1}{m}}(x) \ dx.$$
Since $v^{1/m}\in A_{\infty}$, $v^{\frac{1}{m}}\in A_t$ for some $t>1$ large. Then by the $A_p$ factorization theorem there exist $v_1,v_2\in A_1$ such that $v^{\frac{1}{m}}=v_1v_2^{1-t}$; hence,
$$\nu^{1-p_0}v^\frac{1}{m}=v_1(\nu v_2^{\frac{t-1}{p_0-1}}).$$
By Lemma 2.3 in \cite{CMP} there exists $0<\e <1$, depending only on $[\nu]_{A_1}$, such that $\nu v_2^{\e} \in A_1$ for all $v_2\in A_1$ and $0<\e <\e_0$. Thus, if we let
$$p_0=\frac{2(t-1)}{\e_0}+1,$$
then $\nu v^{\frac{1}{m}}\in A_{p_0}$.

By Muckenhoupt's theorem, $M$ is bounded on $L^{p_0}(\nu ^{1-p_0} v^{\frac{1}{m}})$ and therefore $S$ is bounded on $L^{p_0}(\nu v^{\frac{1}{m}})$ with some constant $C_0$.
Thus by Marcinkiewicz interpolation in the scale of Lorentz spaces, $S$ is bounded on $L^{q,1}(\nu v^{\frac{1}{m}})$ for all $p_0<q<\infty $. In particular, by  \cite[Proposition A.1]{CMP},

$$\| Sf\| _{L^{q,1}(\nu v^{\frac{1}{m}})} \leq 2^{\frac{1}{q}} \Bigg( C_0 \Big(\frac{1}{p_0} - \frac{1}{q}\Big)+ C_1 \Bigg) \|f\|_{L^{q,1}(\nu v^{\frac{1}{m}})}. $$ Thus, for all $q\geq 2p_0$ we have that
$\| Sf\| _{L^{q,1}(\nu v^{\frac{1}{m}})} \leq K_0 \|f\|_{L^{q,1}(\nu v^{\frac{1}{m}})}$ with $K_0=4p_0 \big( C_0 + C_1 \big).$ We emphasize that the constant $K_0$ is valid for every $q\geq 2p_0$.

Again by \cite[Lemma 2.3]{CMP}, for every weight $W_1 \in A_1$ with $[W_1]_{A_1}\leq 2K_0$ there exists $0<\tilde{\e_0} <1$ (that depends only on $K_0$) such that $W_1 W_2^{\e} \in A_1$ for all $W_2 \in A_1$ and $0<\e < \tilde{\e_0}$.

Fix $0<\e < min\{ \tilde{\e_0}, \frac{1}{2p_0} \}$ and let $r=(\frac{1}{\e})'$. Then $r'>2p_0$ and so $S$ is bounded on $L^{r',1}(\nu v^{\frac{1}{m}})$ with constant bounded by $K_0$. Now apply the Rubio de Francia algorithm to define the operator $\mathcal{R}$ on $h\in L^{r',1}(\nu v^{\frac{1}{m}})$, $h\geq 0$, by $$\mathcal{R} h(x)=\sum_{k=0}^{\infty} {\frac{S^kh(x)}{2^kK_0^k}}.$$

It follows immediately from this definition that:

\begin{itemize}
\item $h(x) \leq \mathcal{R} h(x)$;
\item $\| \mathcal{R} h\| _{L^{r',1}(\nu v^{\frac{1}{m}})} \leq 2 \| h\|_{L^{r',1}(\nu v^{\frac{1}{m}})}$;
\item $S(\mathcal{R} h)(x)\leq 2 K_0 \mathcal{R} h(x)$.
\end{itemize}
In particular, it follows from the last item and the definition of $S$ that $\mathcal{R} h \ \nu \in A_1$ with $[\mathcal{R} h \ \nu]_{A_1} \leq 2K_0.$ Let $W_1=\mathcal{R} h \ \nu$ and $W_2=v_1\in A_1$. Then $W_1W_2^{\e}\in A_1$. Hence, $\mathcal{R} h \ \nu \ v^{\frac{1}{mr'}} \in A_1 \subset A_{\infty}$.

Then,
\begin{equation*}
\begin{split}
\Big\| \frac{T(\vec{f})}{v}\Big\|_{L^{\frac{1}{m}, \infty}(\nu v^\frac{1}{m})}^\frac{1}{mr}
& = \sup_{\lambda>0} \ \lambda^{\frac{1}{mr}} \Big(\nu v^\frac{1}{m} \ \{ x\in {\mathbb{R}}^n:\Big|\frac{T(\vec{f)}(x)}{v(x)}\Big|>\lambda \}\Big)^{\frac{1}{r}}\\
& = \sup_{\lambda>0} \ \lambda^{\frac{1}{mr}} \Big(\nu v^\frac{1}{m} \ \{ x\in {\mathbb{R}}^n:\Big|\frac{T(\vec{f)}(x)}{v(x)}\Big|^{\frac{1}{mr}}>\lambda^{\frac{1}{mr}} \}\Big)^{\frac{1}{r}} \\
& = \sup_{t>0} \ t \ \Big(\nu v^\frac{1}{m} \ \{ x\in {\mathbb{R}}^n:\Big|\frac{T(\vec{f)}(x)}{v(x)}\Big|^{\frac{1}{mr}}>t\}\Big)^{\frac{1}{r}} \\
& = \Big\| \Big(\frac{T(\vec{f})}{v}\Big)^\frac{1}{mr}\Big\|_{L^{r, \infty}(\nu v^\frac{1}{m})}\\
& = \sup_{h \in L^{r',1}(\nu v^\frac{1}{m}) \ : \ \|h\|_{L^{r',1}(\nu v^\frac{1}{m})}=1} \Big| \int_{{\mathbb{R}}^n} \Big| \frac{T(\vec{f})(x)}{v(x)} \Big|^{\frac{1}{mr}} \ h(x) \ \nu(x) \ v^\frac{1}{m}(x) \ dx \Big| \\
& = \sup_{h \in L^{r',1}(\nu v^\frac{1}{m}) \ : \ \|h\|_{L^{r',1}(\nu v^\frac{1}{m})}=1} \Big| \int_{{\mathbb{R}}^n} |T(\vec{f})(x)|^{\frac{1}{mr}} \ h(x) \ \nu(x) \ v^\frac{1}{mr'}(x) \ dx \Big|.
\end{split}
\end{equation*}
Before finishing we recall the following fact, which was proved in \cite{LOPTT} (see Corollary 3.8 there), if $w\in A_{\infty}$  and $s>0$ then 

$$ \int_{{\mathbb{R}}^n} |T(\vec{f})(x)|^{s} \ w(x)dx \leq C \int_{{\mathbb{R}}^n} \mathcal{M}(\vec{f})(x)^{s} \ w(x) dx. $$

From the definition of $\mathcal{R}h(x)$, the last inequality and H\"older's inequality

\begin{equation*}
\begin{split}
\int_{{\mathbb{R}}^n} |T(\vec{f})(x)|^{\frac{1}{mr}} \ h(x) \ \nu(x) \ v^\frac{1}{mr'}(x) \ dx
& \leq \int_{{\mathbb{R}}^n} |T(\vec{f})(x)|^{\frac{1}{mr}} \ \mathcal{R}h(x) \ \nu(x) \ v^\frac{1}{mr'}(x) \ dx \\
& \leq \int_{{\mathbb{R}}^n} \mathcal{M}(\vec{f})(x)^{\frac{1}{mr}} \ \mathcal{R}h(x) \ \nu(x) \ v^\frac{1}{mr'}(x) \ dx \\
& = C \int_{{\mathbb{R}}^n} \Big( \frac{\mathcal{M}(\vec{f})(x)}{v(x)} \Big)^{\frac{1}{mr}} \ \mathcal{R}h(x) \ \nu(x) \ v^\frac{1}{m}(x) \ dx \\
& \leq C \ \Big\| \Big(\frac{\mathcal{M}(\vec{f})}{v}\Big)^\frac{1}{mr} \Big\|_{L^{r, \infty}(\nu v^\frac{1}{m})} \ \Big\| \mathcal{R}h \Big\|_{L^{r', 1}(\nu v^\frac{1}{m})} \\
& \leq 2 \ C \ \Big\| \frac{\mathcal{M}(\vec{f})}{v} \Big\|_{L^{\frac{1}{m}, \infty}(\nu v^\frac{1}{m})}^\frac{1}{mr} \ \Big\| h \Big\|_{L^{r', 1}(\nu v^\frac{1}{m})} \\
& = 2 \ C \ \Big\| \frac{\mathcal{M}(\vec{f})}{v} \Big\|_{L^{\frac{1}{m}, \infty}(\nu v^\frac{1}{m})}^\frac{1}{mr}.
\end{split}
\end{equation*}
\\

So we have that
$$\Big\| \frac{T(\vec{f})}{v}\Big\|_{L^{\frac{1}{m}, \infty}(\nu v^\frac{1}{m})}^\frac{1}{mr} \leq C \ \Big\| \frac{\mathcal{M}(\vec{f})}{v}\Big\|_{L^{\frac{1}{m}, \infty}(\nu v^\frac{1}{m})}^\frac{1}{mr}.$$

\end{document}